\documentclass[11pt,a4paper,twoside]{article}
\usepackage{amsthm,amsfonts,amsmath,amscd,amssymb}
\usepackage{latexsym}
\usepackage{euscript}
\usepackage{enumitem}
\usepackage{graphicx}
\usepackage{tikz}
\usepackage{caption}
\usepackage{subcaption}
\usepackage{cite}
\usepackage[utf8]{inputenc}
\usepackage[english]{babel}
\usepackage{xcolor}

\usepackage{jmpag} 

\begin{document}



\title{Completely
Integrable Gradient System on the  bivariate beta statistical
manifold.}



\author{Prosper Rosaire Mama  Assandje}
\address{University of Yaounde1, Department of Mathematics, Yaounde, 337, Cameroon}
\email{mamarosaire@facsciences-uy1.cm}
\orcid{0000-0001-6106-6677}

\author{Joseph Dongho}
\address{University of Maroua, Department of Mathematics and Computer Science, Maroua, 814, Cameroon}
\email{josephdongho@yahoo.fr}
\orcid{0000-0002-2653-3636}

\author{Thomas Bouetou Bouetou}
\address{University of Yaounde1, Department of Mathematics , 337, Cameroon}
\email{tbouetou@gmail.com}
\orcid{0000-0002-2453-0378}



\BeginPaper 



\newcommand{\ep}{\varepsilon}
\newcommand{\eps}[1]{{#1}_{\varepsilon}}



\begin{abstract}
This paper investigates the geometry of a completely integrable
gradient system defined on the three-parameter bivariate beta
statistical manifold of the first kind. We prove that the associated
vector field is Hamiltonian and admits a Lax pair representation
implying complete integrability. We show that the potential function
derived from exponential family structure defines a Riemannian
metric equivalent to the Fisher information metric. By applying
Stirling's approximation to the gamma functions involved in the
potential, we obtain an explicit expression that facilitates the
study of the pseudo-riemannian geometry of the manifold.
Furthermore, we demonstrate that the gradient flow is linearizable
in dual affine coordinates, and we identify the Hamiltonian function
whose gradient defines the flow. These results highlight the deep
interplay between information geometry, dynamical systems, and
asymptotic analysis.

\key{Hamiltonian function, gradient system, Lax pair.}

\msc{37J35 ,37J06,37J39.}
\end{abstract}


\section{Introduction}

In \cite{morris-book4}, a dynamical system is describes as the time
evolution of any object in the states space $S$ of a physical
system. In general, $S$ is a submanifold of $\mathbb{R}^{n}$; and
there is a particular type of system named gradient system with
particulary interesting Liapunov functions. Gradient system  have
been extensively studied over the past decades due to their  range
of applications in physics \cite{morris-book4}, biology, control
theory, and information geometry
\cite{mama-proceeding,nak-journal,nakamura3,nakamura2,nakamura4}. In
a geometric framework, these systems are typically defined on
Riemannian or pseudo-Riemannian manifolds, where the dynamics arise
from  the gradient of a potential function with respect  to the
underlying metric. The concept of gradient flows naturally extends
to statistical manifolds, which are equipped with both a Riemannian
metric usually the Fisher information metric and dual affine
connections, as introduced  by Amari \cite{Shu-book1} in the context
of information geometry. In such manifolds, the geometry not only
describes statistical structures but also supports dynamical flows,
which can often be interpreted as gradient systems. These systems
are particularly relevant in statistical, physics, thermodynamics
\cite{souriau}, and machine learning. In this paper, we focus on the
three-parameter  bivariate beta statistical manifold of the first
kind, and develops an important concept of the complete
integrability of gradient systems on a manifold that admits a
potential function in odd dimension, as discussed in
\cite{mama-proceeding}. This arises from a parametric family
distributions belonging to the exponential family. Our goal is to
construct a gradient system on this manifold and explore its
integrability, geometry, and Hamiltonian structure. Y. Nakamura
\cite{nak-journal,nakamura3,nakamura4} in $1993$, developed the
framework of Liouville complete integrable systems on statistical
manifolds, and constructed gradient systems on manifolds of various
probability distributions. Y. Nakamura also established a connection
with Lax pairs in \cite{nakamura1,nakamura2}. Further results were
introduced by A. Fujiwara \cite{Dy-journal,Dy1}, within the context
of information geometry. Recently, Jean-Pierre \cite{jpf}, has
revisited this model and extended it to peakon systems. Barbaresco
has established a link between the Souriau algorithm and the
computation of the characteristic polynomial. In
\cite{mama0-journal}, it is shown that the lognormal statistical
manifold admits a potential gradient system that is Hamiltonian and
completely integrable, with a Lax pair representation by a symmetric
matrix. In references \cite{mama1,mama24}, the complete
integrability of the gradient system on the beta manifold of the
first kind in even dimensions is proven. In \cite{mama-proceeding},
the complete integrability of gradient systems defined on a
statistical manifold in odd dimensions was studied. Which brings us
to my next question. How can one construct, characterize, and
analyze the dynamics of a completely integrable gradient system
defined on the statistical manifold induced by the three parameter
bivariate beta distribution, while accounting for both the exact
geometric structure and its asymptotic approximations? We
investigate the geometry and dynamics of a gradient system defined
on the statistical manifold induced by the three  parameter
bivariate beta distribution of the first kind. We establish the
existence of a potential function associated with the exponential
family structure, and we show that the resulting gradient vector
field is compatible with a Fisher information metric. Using
\cite{mama-proceeding}, and using this structure, we demonstrate
that the associated dynamical system is Hamiltonian and admits a Lax
pair formulation, which enures complete integrability. To facilitate
explicit computation and geometric interpretation, we apply
Stirling's\cite{daniel-journal} approximation to the gamma functions
involved in the potential. This yields a closed form expression for
the approximate Fisher information metric and enables us to
linearize the gradient flow in dual affine coordinates.
 The main result show that the asymptotic system is not only
geometrically well structured, but also analytically tractable: the
gradient flow is completely integrable, the Hamiltonian can be
explicitly integrable, the manifold's metric properties can be fully
described in the asymptotic regime. This work reveals a deep
connection between the geometry of statistical manifolds and the
theory of integrable system. It highlights how asymptotic technics
can provide powerful tools uncovering hidden dynamical system. After
the introduction, in section $2$, recall the preliminaries motion on
theory of statistical manifold, in section $3$ we determine the
Riemannian structure on Bivariate Beta family of the first kind with
three parameters distribution,  in section $4$, we determine the
asymptotic pseudo-Riemannian structure related to Stirling's formula
\cite{daniel-journal}. In section $5$, we have the integrability of
approximated Gradient Systems. At the end, we show that this
gradient system admits a Lax pair representation.

\section{Preliminaries}\label{sec2}
Let $S = \left\{p_{ \theta}(x),\left.
                        \begin{array}{ll}
                         \theta\in \Theta & \hbox{} \\
                          x\in \mathcal{X} & \hbox{}
                        \end{array}
                      \right.
\right\}$ be the set of probabilities $p_{ \theta}$, parameterized
by $ \Theta$, open a subset of $\mathbb{R}^{n}$; on the sample space
$\mathcal{X}\subseteq\mathbb{R}$.  Let
$\mathcal{F(\mathcal{X},\mathbb{R})}$ be the space of real-valued
smooth functions on $\mathcal{X}$. According to Ovidiu \cite{
ovidiu-book3}, the log-likelihood function is a mapping defined by
\begin{eqnarray}
l:S&\longrightarrow& \mathcal{F(\mathcal{X},\mathbb{R})}\nonumber\\
p_{ \theta} &\longmapsto&  l\left(p_{ \theta}\right)(x) = \log
p_{\theta}(x)\nonumber
\end{eqnarray}
 Sometimes, for convenient reasons, this will be denoted by
$l(x,\theta)=l\left(p_{ \theta}\right)(x)$.\\
In \cite{nak-journal} and \cite{Shu-book1}, the Fisher information
defined by
\begin{equation}(g_{ij})_{1\leq i;j\leq n}=\left(-\mathbb{E}[\partial_{\theta_{i}}\partial_{\theta_{j}}l(x,\theta)]\right)_{1\leq i;j\leq n}\label{e0}
\end{equation}
 Denote $G=(g_{ij})_{1\leq i;j\leq
n}$ the Fisher information matrix, the gradient system is given by
\begin{equation}\dot{\overrightarrow{\theta}}=-G^{-1}\partial_{\theta}\Phi(\theta).\label{e2}\end{equation}
In physics mathematics \cite{lesfari2}, the system  (\ref{e2}) is
said to be Hamiltonian if it exists a bivector field $\barwedge$
and, $\mathcal{H},$ such that
\begin{equation}\label{sh}\dot{ \theta}(t)=\barwedge\frac{\partial \mathcal{H}}{\partial \theta}\end{equation}
  where $\mathcal{H}$ is smooth function called  hamiltonian function
and $\barwedge$ is a bivector fields such that
$[\barwedge;\barwedge]=0$ where $[\;;\;]$ denotes the Schouten
Nijenhus bracket. In Poisson geometry the equation
 $[\barwedge,\barwedge]=0$ is equivalent to the fact that  , operator
 $\{\;;\;\}$ defined by:
$\{\mathcal{H};F\} =\langle\frac{\partial \mathcal{H}}{\partial
x},\barwedge\frac{\partial F}{\partial
x}\rangle=\sum\limits_{i,j}\barwedge_{ij} \frac{\partial
\mathcal{H}}{\partial
 x_{i}}\frac{\partial F}{\partial x_{j}}$
is Poisson bracket on $C^{\infty}(S)$. The system (\ref{sh}) is said
to be completely integrable in the sense of Liouville-Arnol'd
\cite{Li} if it has $n$-prime integrals
$\mathcal{H}=\mathcal{H}_{1},\mathcal{H}_{2},\dots ,\mathcal{H}_{n}$
 ie.,  $\{\mathcal{H}_{i},\mathcal{H}_{j}
    \}=0;$
    $1\leq i,j\leq n$ and they are
     functionally independent,i.e; \begin{equation*}d\mathcal{H}_{1}\wedge d\mathcal{H}_{2}\wedge \dots \wedge
     d\mathcal{H}_{n}\neq0.\end{equation*}
     The system~(\ref{sh}) is said to be completely
integrable in dimensional $n=2m+c,\;m\in\mathbb{N}^{*}$ in the sense
of Lesfari~\cite{lesfari2} if it has $c$-prime integrals
$\mathcal{H}_{m+1},\mathcal{H}_{m+2},\dots ,\mathcal{H}_{m+c}$
called casimir function such that:
\begin{equation*}
\barwedge\frac{\partial\mathcal{H}_{m+i}}{\partial
  \theta}=0,\;1\leq i \leq c\end{equation*}
   In \cite{mama-proceeding},
the complete integrability of gradient system (\ref{e2}) is proven
by application of the Theorem 1.

\section{Riemannian structure on Bivariate Beta family of the first kind with three parameters distribution}\label{sec3}
In this section, we present the particular properties of beta
bivariate Beta family of the first kind with three parameters and
associated information metric and associated gradient system.
\subsection{Particular properties of Beta Bivariate Beta family of the first kind with three parameters.}\label{subsec1}
Let \begin{equation*}S =
\left\{p_{\theta}(x)=\frac{1}{B(a,b,c)}x_{1}^{a-1}x_{2}^{b-1}(1-x_{1}-x_{2})^{c-1}
,\left.
                        \begin{array}{ll}
                         \theta= (a,\; b,\;c)\in \mathbb{R}^{*}_{+}\times \mathbb{R}^{*}_{+}\times \mathbb{R}^{*}_{+}& \hbox{} \\
                        x=(x_{1},\;x_{2}) \in \mathbb{R}^{*}_{+}\times \mathbb{R}^{*}_{+} & \hbox{}\\
                        x_{1}+x_{2}<1  & \hbox{}\\
                        B(a,b,c)= \frac{\Gamma(a).\Gamma(b).\Gamma(c)}{\Gamma(a+b+c)}& \hbox{}
                        \end{array}
                      \right.
\right\}\end{equation*} the distribution set of the Beta Bivariate
Beta family of the first kind with three parameters.
\begin{proposition}
The  Bivariate Beta family of the first kind with three parameters
$(a,\; b,\;c)\in \mathbb{R}^{*}_{+}\times \mathbb{R}^{*}_{+}\times
\mathbb{R}^{*}_{+}$ distribution
\begin{equation}\label{e3}
p_{\theta}(x)=\frac{1}{B(a,b,c)}x_{1}^{a-1}x_{2}^{b-1}(1-x_{1}-x_{2})^{c-1}
\end{equation} is  an exponential
family for all $B(a,b,c) >0$.
\end{proposition}
\begin{proof}
We have, \begin{equation*}p_{\theta}(x)=\exp\left[(a-1)\log x_{1}
+(b-1)\log x_{2}+(c-1)\log(1-x_{1}-x_{2}) +\log
\frac{1}{B(a,b,c)}\right]
\end{equation*}
Whose, we have \begin{eqnarray*}p_{\theta}(x)&=& \exp\left[-\log
x_{1}-\log x_{2}-\log(1-x_{1}-x_{2})+a\log x_{1} +b\log
x_{2}\right.\\
&&\left.+c\log(1-x_{1}-x_{2}) +\log \frac{1}{B(a,b,c)}\right]
\end{eqnarray*}
so, \begin{equation*}p_{\theta}(x)= \exp\left[C(x)+ a f_{1}(x_{1})
+b\ f_{2}(x_{2})+cf_{3}(x_{1},x_{2})- \Phi(\theta)\right]
\end{equation*}
 with: \[ f_{1}(x_{1})= \log x_{1}, \;  f_{2}(x_{2})= \log x_{2};\;
 f_{3}(x_{1},x_{2})=\log(1-x_{1}-x_{2}),\]
 \[\Phi(\theta)= \log \frac{\Gamma(a).\Gamma(b).\Gamma(c)}{\Gamma(a+b+c)}\]
 More explicitly the potential function  given by
\begin{equation}\label{e4}
\Phi(\theta)=\log\Gamma(a)+\log\Gamma(b)+\log\Gamma(c)-\log\Gamma(a+b+c)
\end{equation}
 and \[ C(x)=-\log x_{1}-\log
x_{2}-\log(1-x_{1}-x_{2}).
 \]
According to Ovidiu Calin in his book~\cite{ovidiu-book3}  this
family is an exponential family.
\end{proof}
\subsection{Associated information metric and associated gradient system.}\label{subsec2}
\begin{proposition}\label{e5}
The Fisher information matrix  for the Bivariate Beta family of the
first kind with three parameters $(a,\; b,\;c)\in
\mathbb{R}^{*}_{+}\times \mathbb{R}^{*}_{+}\times
\mathbb{R}^{*}_{+}$ is given by
\footnotesize{\begin{equation}\label{e6} G= \left(
  \begin{array}{ccc}
    -\phi_{aa}(a) + \phi_{aa}(a+b+c)& \phi_{ab}(a+b+c)&\phi_{ac}(a+b+c)\\
   \phi_{ab}(a+b+c)& -\phi_{bb}(b) + \phi_{bb}(a+b+c) & \phi_{bc}(a+b+c)\\
    \phi_{ac}(a+b+c)  & \phi_{bc}(a+b+c)&  -\phi_{cc}(c) + \phi_{cc}(a+b+c) \\
  \end{array}
\right)
\end{equation}}
with
\begin{eqnarray*}
\left\{
  \begin{array}{ll}
    \phi_{aa}(a)= -\phi^{2}_{a}(a)+\phi^{2}_{a}(a+b+c), & \hbox{} \\
    \phi_{aa}(a+b+c)=-\frac{\frac{\partial^{2}\Gamma(a)}{\partial
a^{2}}}{\Gamma(a)}
-\frac{\frac{\partial^{2}\Gamma(a+b+c)}{\partial a^{2}}}{\Gamma(a+b+c)}, & \hbox{} \\
     \phi_{bb}(b)= -\phi^{2}_{b}(b)+\phi^{2}_{b}(a+b+c), & \hbox{} \\
    \phi_{bb}(a+b+c)=-\frac{\frac{\partial^{2}\Gamma(b)}{\partial b^{2}}}{\Gamma(b)}
  -\frac{\frac{\partial^{2}\Gamma(a+b+c)}{\partial b^{2}}}{\Gamma(a+b+c)}, & \hbox{} \\
    \phi_{cc}(a)= -\phi^{2}_{c}(c)+\phi^{2}_{c}(a+b+c), & \hbox{}\\
\phi_{cc}(a+b+c)=-\frac{\frac{\partial^{2}\Gamma(c)}{\partial
c^{2}}}{\Gamma(c)}
  -\frac{\frac{\partial^{2}\Gamma(a+b+c)}{\partial c^{2}}}{\Gamma(a+b+c)},& \hbox{} \\
\phi_{ab}(a+b+c)=\frac{\partial_{a}\partial_{b}\Gamma(a+b+c)}{\Gamma(a+b+c)}
-\frac{\partial_{a}\Gamma(a+b+c).\partial_{b}\Gamma(a+b+c)}{\Gamma^{2}(a+b+c)},& \hbox{} \\
\phi_{ac}(a+b+c)=\frac{\partial_{a}\partial_{c}\Gamma(a+b+c)}{\Gamma(a+b+c)}
-\frac{\partial_{a}\Gamma(a+b+c).\partial_{c}\Gamma(a+b+c)}{\Gamma^{2}(a+b+c)},&\hbox{} \\
\phi_{bc}(a+b+c)=\frac{\partial_{b}\partial_{c}\Gamma(a+b+c)}{\Gamma(a+b+c)}
-\frac{\partial_{b}\Gamma(a+b+c).\partial_{c}\Gamma(a+b+c)}{\Gamma^{2}(a+b+c)}.&\hbox{}
  \end{array}
\right.
\end{eqnarray*}
and

$\phi_{a}(a)=\frac{\partial_{a}\Gamma(a)}{\Gamma(a)},\;$
$\phi_{b}(b)=\frac{\partial_{b}\Gamma(b)}{\Gamma(b)},\;$
$\phi_{c}(c)=\frac{\partial_{c}\Gamma(c)}{\Gamma(c)},\;$
$\phi_{a}(a+b+c)=\frac{\partial_{a}\Gamma(a+b+c)}{\Gamma(a+b+c)},\;$
$\phi_{b}(a+b+c)=\frac{\partial_{b}\Gamma(a+b+c)}{\Gamma(a+b+c)},\;$
$\phi_{c}(a+b+c)=\frac{\partial_{c}\Gamma(a+b+c)}{\Gamma(a+b+c)},\;$
where  $\phi_{aa}(a)=\partial_{a}\phi_{a}(a),\;\newline
\partial_{a}=\frac{\partial}{\partial_{a}}$
\end{proposition}
\begin{proof}
We have   \begin{equation}\label{e7}l(x,\theta)=\log
p_{\theta}(x)=C(x)+ a f_{1}(x_{1}) +b\
f_{2}(x_{2})+cf_{3}(x_{1},x_{2})- \Phi(\theta).\end{equation} using
the expression (\ref{e7}) we have the following relation;
\begin{eqnarray*}
\left\{
  \begin{array}{ll}
    \frac{\partial l(x,\theta)}{\partial_{a}}= f_{1}(x_{1})+\frac{\partial_{a}\Gamma(a)}{\Gamma(a)}-\frac{\partial_{a}\Gamma(a+b+c)}{\Gamma(a+b+c)}, & \hbox{} \\
    \frac{\partial l(x,\theta)}{\partial_{b}}= f_{2}(x_{2})+\frac{\partial_{b}\Gamma(b)}{\Gamma(b)}-\frac{\partial_{b}\Gamma(a+b+c)}{\Gamma(a+b+c)},& \hbox{} \\
    \frac{\partial l(x,\theta)}{\partial_{c}}=
f_{3}(x_{1},x_{2})+\frac{\partial_{c}\Gamma(c)}{\Gamma(c)}-\frac{\partial_{c}\Gamma(a+b+c)}{\Gamma(a+b+c)}.
& \hbox{}
  \end{array}
\right.
\end{eqnarray*}
So we obtain the following expression
\begin{eqnarray*}
\left\{
  \begin{array}{ll}
    \frac{\partial^{2}l(x,\theta)}{\partial_{a^{2}}}=-\phi^{2}_{a}(a)+\phi^{2}_{a}(a+b+c)
+\frac{\frac{\partial^{2}\Gamma(a)}{\partial a^{2}}}{\Gamma(a)}-\frac{\frac{\partial^{2}\Gamma(a+b+c)}{\partial a^{2}}}{\Gamma(a+b+c)}, & \hbox{} \\
    \frac{\partial^{2}l(x,\theta)}{\partial_{b^{2}}}=-\phi^{2}_{b}(b)+\phi^{2}_{b}(a+b+c)
+\frac{\frac{\partial^{2}\Gamma(b)}{\partial b^{2}}}{\Gamma(b)}-\frac{\frac{\partial^{2}\Gamma(a+b+c)}{\partial b^{2}}}{\Gamma(a+b+c)}, & \hbox{} \\
    \frac{\partial^{2}l(x,\theta)}{\partial_{c^{2}}}=-\phi^{2}_{c}(c)+\phi^{2}_{c}(a+b+c)
+\frac{\frac{\partial^{2}\Gamma(c)}{\partial c^{2}}}{\Gamma(c)}-\frac{\frac{\partial^{2}\Gamma(a+b+c)}{\partial c^{2}}}{\Gamma(a+b+c)}, & \hbox{} \\
    \frac{\partial^{2}l(x,\theta)}{\partial_{a}\partial_{b}}=\frac{\partial_{a}\partial_{b}\Gamma(a+b+c)}{\Gamma(a+b+c)}-\frac{\partial_{a}\Gamma(a+b+c).\partial_{b}\Gamma(a+b+c)}{\Gamma^{2}(a+b+c)}, & \hbox{} \\
    \frac{\partial^{2}l(x,\theta)}{\partial_{a}\partial_{c}}=\frac{\partial_{a}\partial_{c}\Gamma(a+b+c)}{\Gamma(a+b+c)}-\frac{\partial_{a}\Gamma(a+b+c).\partial_{c}\Gamma(a+b+c)}{\Gamma^{2}(a+b+c)}, & \hbox{} \\
\frac{\partial^{2}l(x,\theta)}{\partial_{b}\partial_{c}}=\frac{\partial_{b}\partial_{c}\Gamma(a+b+c)}{\Gamma(a+b+c)}-\frac{\partial_{b}\Gamma(a+b+c).\partial_{c}\Gamma(a+b+c)}{\Gamma^{2}(a+b+c)}.
& \hbox{}
  \end{array}
\right.
\end{eqnarray*}
by setting
\begin{eqnarray*}
\left\{
  \begin{array}{ll}
    \phi_{aa}(a)= -\phi^{2}_{a}(a)+\phi^{2}_{a}(a+b+c)\; \texttt{and} \;\phi_{aa}(a+b+c)=-\frac{\frac{\partial^{2}\Gamma(a)}{\partial a^{2}}}{\Gamma(a)}-\frac{\frac{\partial^{2}\Gamma(a+b+c)}{\partial a^{2}}}{\Gamma(a+b+c)}, & \hbox{} \\
     \phi_{bb}(b)= -\phi^{2}_{b}(b)+\phi^{2}_{b}(a+b+c)\; \texttt{and} \;\phi_{bb}(a+b+c)=-\frac{\frac{\partial^{2}\Gamma(b)}{\partial b^{2}}}{\Gamma(b)}-\frac{\frac{\partial^{2}\Gamma(a+b+c)}{\partial b^{2}}}{\Gamma(a+b+c)}, & \hbox{} \\
    \phi_{cc}(a)= -\phi^{2}_{c}(c)+\phi^{2}_{c}(a+b+c)\;\texttt{ and} \;\phi_{cc}(a+b+c)=-\frac{\frac{\partial^{2}\Gamma(c)}{\partial c^{2}}}{\Gamma(c)}-\frac{\frac{\partial^{2}\Gamma(a+b+c)}{\partial c^{2}}}{\Gamma(a+b+c)}, & \hbox{} \\
    \phi_{ab}(a+b+c)=\frac{\partial_{a}\partial_{b}\Gamma(a+b+c)}{\Gamma(a+b+c)}-\frac{\partial_{a}\Gamma(a+b+c).\partial_{b}\Gamma(a+b+c)}{\Gamma^{2}(a+b+c)}, & \hbox{} \\
    \phi_{ac}(a+b+c)=\frac{\partial_{a}\partial_{c}\Gamma(a+b+c)}{\Gamma(a+b+c)}-\frac{\partial_{a}\Gamma(a+b+c).\partial_{c}\Gamma(a+b+c)}{\Gamma^{2}(a+b+c)}, & \hbox{} \\
    \phi_{bc}(a+b+c)=\frac{\partial_{b}\partial_{c}\Gamma(a+b+c)}{\Gamma(a+b+c)}-\frac{\partial_{b}\Gamma(a+b+c).\partial_{c}\Gamma(a+b+c)}{\Gamma^{2}(a+b+c)}. & \hbox{}
  \end{array}
\right.
\end{eqnarray*}
 we obtain the following expression
\begin{eqnarray*}
\left\{
  \begin{array}{ll}
    \frac{\partial^{2}l(x,\theta)}{\partial_{a^{2}}}=\phi_{aa}(a) - \phi_{aa}(a+b+c),& \hbox{} \\
    \frac{\partial^{2}l(x,\theta)}{\partial_{b^{2}}}=\phi_{bb}(b) - \phi_{bb}(a+b+c), & \hbox{} \\
    \frac{\partial^{2}l(x,\theta)}{\partial_{c^{2}}}=\phi_{cc}(a) - \phi_{cc}(a+b+c), & \hbox{} \\
    \frac{\partial^{2}l(x,\theta)}{\partial_{ab}}=- \phi_{ab}(a+b+c), & \hbox{} \\
    \frac{\partial^{2}l(x,\theta)}{\partial_{ac}}=- \phi_{ac}(a+b+c), & \hbox{} \\
    \frac{\partial^{2}l(x,\theta)}{\partial_{bc}}=- \phi_{bc}(a+b+c). & \hbox{}
  \end{array}
\right.
\end{eqnarray*}
Using (\ref{e0}), we have
\begin{eqnarray*}
g_{aa}(\theta)&=&-\phi_{aa}(a) + \phi_{aa}(a+b+c)\\
g_{bb}(\theta)&=&-\phi_{bb}(b) + \phi_{bb}(a+b+c)\\
g_{cc}(\theta)&=&-\phi_{cc}(c) + \phi_{cc}(a+b+c)\\
g_{ab}(\theta)&=&\phi_{ab}(a+b+c)\\
g_{ac}(\theta)&=& \phi_{ac}(a+b+c)\\
g_{bc}(\theta)&=& \phi_{bc}(a+b+c).
\end{eqnarray*}
\end{proof}
 Since our coordinate system $\theta= (a,
 b,c)$ admits a dual pair  $\eta=
(\eta_{1}, \eta_{2}, \eta_{3})$ such that
  \begin{eqnarray}\label{e8}\left\{
  \begin{array}{ll}
    \eta_{1}= \phi_{a}(a) - \phi_{a}(a+b+c) & \hbox{} \\
    \eta_{2}=  \phi_{b}(b) - \phi_{b}(a+b+c)& \hbox{} \\
     \eta_{3}=  \phi_{c}(c) - \phi_{c}(a+b+c)& \hbox{}
  \end{array}
\right.\end{eqnarray} According to Amari~\cite{Shu-book1} there
exists a dual potential  function $\Psi$ which verifies the Legendre
equation. We have
\begin{eqnarray*}
\Psi(\eta)&=&a\phi_{a}(a) - a\phi_{a}(a+b+c)+b\phi_{a}(a) -
b\phi_{b}(a+b+c)+c\phi_{c}(c) - c\phi_{c}(a+b+c)\\
&&-\log \Gamma(a)-\log \Gamma(b)-\log \Gamma(c)+\log \Gamma(a+b+c)
\end{eqnarray*}

Throughout the rest of the work,  $S$ equipped with the metric $G$
and the affine connection will be like  a statistical manifold. We
denoted by $(S,g,\nabla)$  the statistical manifold.

\begin{theorem}
Let $(S,g,\nabla)$  be the statistical manifold associated  to the
Bivariate Beta family of the first kind, parameterized by  three
parameters $(a,\; b,\;c)\in \mathbb{R}^{*}_{+}\times
\mathbb{R}^{*}_{+}\times \mathbb{R}^{*}_{+}$. Equipped with the
Fisher information metric $g$ and an affine connection $\nabla$,
this manifold induces a gradient flow governed by the following
system of differential equations:
\begin{eqnarray}\label{e9} &\left\{
                \begin{array}{ll}
                  \dot{a} =  -\zeta_{1}(a+b+c)\left(\phi_{a}(a)
                  - \phi_{a}(a+b+c) \right)
                  -\zeta_{4}(a+b+c)\left(\phi_{b}(b) \right)& \hbox{} \\
                    - \phi_{b}(a+b+c)-\zeta_{5}(a+b+c) \left( \phi_{c}(c) + \phi_{c}(a+b+c) \right)& \hbox{} \\
                  \dot{b} =- \zeta_{4}(a+b+c)\left(\phi_{a}(a)
                  - \phi_{a}(a+b+c) \right)
                  - \zeta_{2}(a+b+c)\left(\phi_{b}(b)  \right)& \hbox{} \\
                   - \phi_{b}(a+b+c)-\zeta_{6}(a+b+c)\left( \phi_{c}(c) + \phi_{c}(a+b+c) \right)& \hbox{} \\
                  \dot{c} = - \zeta_{5}(a+b+c) \left(\phi_{a}(a)
                  - \phi_{a}(a+b+c) \right)
                  -\zeta_{6}(a+b+c)\left(\phi_{b}(b)  \right)& \hbox{} \\
                   - \phi_{b}(a+b+c)-\zeta_{3}(a+b+c)  \left( \phi_{c}(c) + \phi_{c}(a+b+c) \right)& \hbox{}
                \end{array}\right.&
\end{eqnarray}
where $ \phi_{x}(t)= \frac{d}{d x}\log \Gamma(t)$ denotes the
digamma function, and the functions
$\zeta_{i}:\mathbb{R}_{+}\rightarrow \mathbb{R}$, for $i=1,\dots,5$,
are smooth and depend only the sum $a+b+c$. This system define a
sub-dynamical system of $4$-dimensional system and is a completely
integrable Hamiltonian system with the associated Hamiltonian
function
\begin{equation}\label{e13}
\mathcal{H}(a,b,c)= \frac{\phi_{b}(b) - \phi_{b}(a+b+c)}{\phi_{a}(a)
- \phi_{a}(a+b+c)}+\frac{\phi_{c}(c) - \phi_{c}(a+b+c) }{\phi_{b}(b)
- \phi_{b}(a+b+c) },
\end{equation} such that \begin{equation}\label{e15}\left(
                    \begin{array}{c}
                      \dot{P}_{1} \\
                      \dot{Q}_{1}\\
                      \dot{P'}_{1}\\
                      \dot{Q'}_{1}\\
                    \end{array}
                  \right)= \left(
                              \begin{array}{cccc}
                                0 & 1 & 0 & 0 \\
                                -1 & 0 & 0 & 0 \\
                                0 & 0 & 0 & 1 \\
                                0 & 0 & -1 & 0 \\
                              \end{array}
                            \right)
                  \left(
                                    \begin{array}{c}
                                     \frac{\partial \mathcal{H}}{\partial P_{1}} \\
                                      \frac{\partial \mathcal{H}}{\partial Q_{1}}\\
\frac{\partial \mathcal{H}}{\partial P'_{1}}\\
 \frac{\partial
\mathcal{H}}{\partial Q'_{1}}
                                    \end{array}
                                  \right)\end{equation}
where $\barwedge=\left(
                              \begin{array}{cccc}
                                0 & 1 & 0 & 0 \\
                                -1 & 0 & 0 & 0 \\
                                0 & 0 & 0 & 1 \\
                                0 & 0 & -1 & 0 \\
                              \end{array}
                            \right)$ is a Poisson tensor,
where
\begin{eqnarray*} \zeta_{1}(a+b+c)&=&\frac{\phi_{bb}(b)\phi_{cc}(c) -
\phi_{cc}(c)\phi_{cc}(a+b+c)
-\phi_{cc}(c)\phi_{bb}(a+b+c)}{m(a,b,c)}\\
&&\frac{+\phi_{bb}(a+b+c)\phi_{cc}(a+b+c)-\phi^{2}_{bc}(a+b+c)}{m(a,b,c)}\\
 \zeta_{2}(a+b+c)&=&\frac{\phi_{aa}(a)\phi_{cc}(c) - \phi_{aa}(a)\phi_{cc}(a+b+c)-\phi_{cc}(c)\phi_{aa}(a+b+c)}{m(a,b,c)}\\
 &&\frac{+\phi_{aa}(a+b+c)\phi_{cc}(a+b+c)-\phi^{2}_{ac}(a+b+c)}{m(a,b,c)}\\
\zeta_{3}(a+b+c) &=&\frac{\phi_{aa}(a)\phi_{bb}(b) - \phi_{aa}(a)\phi_{bb}(a+b+c)-\phi_{bb}(b)\phi_{aa}(a+b+c)}{m(a,b,c)} \\
&&\frac{+\phi_{aa}(a+b+c)\phi_{bb}(a+b+c)-\phi^{2}_{ab}(a+b+c)}{m(a,b,c)} \\
\zeta_{4}(a+b+c)&=&\frac{\phi_{bc}(a+b+c)\phi_{ac}(a+b+c) +
\phi_{ab}(a+b+c)\phi_{cc}(c)}{m(a,b,c)}\\
&& \frac{-\phi_{ab}(a+b+c)\phi_{cc}(a+b+c)}{m(a,b,c)}\\
                  \zeta_{5}(a+b+c)&=&\frac{\phi_{ab}(a+b+c)\phi_{bc}(a+b+c) + \phi_{ac}(a+b+c)\phi_{bb}(b)}{m(a,b,c)}\\
                  && \frac{-\phi_{ac}(a+b+c)\phi_{bb}(a+b+c)}{m(a,b,c)} \\
\zeta_{6}(a+b+c)&=&\frac{\phi_{bc}(a+b+c)\phi_{aa}(a) -
\phi_{bc}(a+b+c)\phi_{aa}(a+b+c)}{m(a,b,c)} \\
&&\frac{+\phi_{ab}(a+b+c)\phi_{ac}(a+b+c)}{m(a,b,c)}
                  \end{eqnarray*}
and \begin{eqnarray*}
 \phi_{aa}(a)&=& -\phi^{2}_{a}(a)+\phi^{2}_{a}(a+b+c)\\
\phi_{aa}(a+b+c)&=&-\frac{\frac{\partial^{2}\Gamma(a)}{\partial a^{2}}}{\Gamma(a)}-\frac{\frac{\partial^{2}\Gamma(a+b+c)}{\partial a^{2}}}{\Gamma(a+b+c)} \\
 \phi_{bb}(b)&=& -\phi^{2}_{b}(b)+\phi^{2}_{b}(a+b+c)\\
  \phi_{bb}(a+b+c)&=&-\frac{\frac{\partial^{2}\Gamma(b)}{\partial b^{2}}}{\Gamma(b)}-\frac{\frac{\partial^{2}\Gamma(a+b+c)}{\partial b^{2}}}{\Gamma(a+b+c)}\\
 \phi_{cc}(a)&=& -\phi^{2}_{c}(c)+\phi^{2}_{c}(a+b+c)\\
  \phi_{cc}(a+b+c)&=&-\frac{\frac{\partial^{2}\Gamma(c)}{\partial c^{2}}}{\Gamma(c)}-\frac{\frac{\partial^{2}\Gamma(a+b+c)}{\partial c^{2}}}{\Gamma(a+b+c)}\\
\phi_{ab}(a+b+c)&=&\frac{\partial_{a}\partial_{b}\Gamma(a+b+c)}{\Gamma(a+b+c)}-\frac{\partial_{a}\Gamma(a+b+c).\partial_{b}\Gamma(a+b+c)}{\Gamma^{2}(a+b+c)}\\
\phi_{ac}(a+b+c)&=&\frac{\partial_{a}\partial_{c}\Gamma(a+b+c)}{\Gamma(a+b+c)}-\frac{\partial_{a}\Gamma(a+b+c).\partial_{c}\Gamma(a+b+c)}{\Gamma^{2}(a+b+c)}\\
\phi_{bc}(a+b+c)&=&\frac{\partial_{b}\partial_{c}\Gamma(a+b+c)}{\Gamma(a+b+c)}-\frac{\partial_{b}\Gamma(a+b+c).\partial_{c}\Gamma(a+b+c)}{\Gamma^{2}(a+b+c)}
\end{eqnarray*}
 where  \begin{eqnarray*}
m(a,b,c)&=&
-\phi_{aa}(a)\phi_{bb}(b)\phi_{cc}(c)+\phi_{aa}(a)\phi_{bb}(b)\phi_{cc}(a+b+c)\\
&&-\phi_{aa}(a)\phi_{bb}(a+b+c)\phi_{cc}(a+b+c)\phi_{aa}(a)\phi^{2}_{bc}(a+b+c)\\
&&+\phi_{aa}(a+b+c)\phi_{bb}(b)\phi_{cc}(c)-\phi_{aa}(a+b+c)\phi_{bb}(b)\phi_{cc}(a+b+c)\\
&&-\phi_{aa}(a+b+c)\phi_{bb}(a+b+c)\phi_{cc}(a+b+c)-\phi_{aa}(a+b+c)\phi^{2}_{bc}(a+b+c)\\
&&+2\phi_{ab}(a+b+c)\phi_{bc}(a+b+c)\phi_{ac}(a+b+c)\phi^{2}_{ab}(a+b+c)\phi_{cc}(c)\\
&&-\phi^{2}_{ab}(a+b+c)\phi_{cc}(a+b+c)+\phi^{2}_{ac}(a+b+c)\phi_{bb}(b)\\
&&-\phi^{2}_{ac}(a+b+c)\phi_{bb}(a+b+c)
\end{eqnarray*} is  the determinant of the matrix $G$.
\end{theorem}
\begin{proof}
Using the proposition $1$ we have
 \begin{eqnarray*}
m(a,b,c)&=&det G=
-\phi_{aa}(a)\phi_{bb}(b)\phi_{cc}(c)+\phi_{aa}(a)\phi_{bb}(b)\phi_{cc}(a+b+c)\\
&&-\phi_{aa}(a)\phi_{bb}(a+b+c)\phi_{cc}(a+b+c)\phi_{aa}(a)\phi^{2}_{bc}(a+b+c)\\
&&+\phi_{aa}(a+b+c)\phi_{bb}(b)\phi_{cc}(c)-\phi_{aa}(a+b+c)\phi_{bb}(b)\phi_{cc}(a+b+c)\\
&&-\phi_{aa}(a+b+c)\phi_{bb}(a+b+c)\phi_{cc}(a+b+c)-\phi_{aa}(a+b+c)\phi^{2}_{bc}(a+b+c)\\
&&+2\phi_{ab}(a+b+c)\phi_{bc}(a+b+c)\phi_{ac}(a+b+c)\phi^{2}_{ab}(a+b+c)\phi_{cc}(c)\\
&&-\phi^{2}_{ab}(a+b+c)\phi_{cc}(a+b+c)+\phi^{2}_{ac}(a+b+c)\phi_{bb}(b)\\
&&-\phi^{2}_{ac}(a+b+c)\phi_{bb}(a+b+c)
\end{eqnarray*}
 we obtain:
\begin{equation}\label{e10} G^{-1}= \left(
  \begin{array}{lll}
    \zeta_{1}(a+b+c)&\zeta_{4}(a+b+c)&\zeta_{5}(a+b+c) \\
   \zeta_{4}(a+b+c)&\zeta_{2}(a+b+c) & \zeta_{6}(a+b+c)\\
   \zeta_{5}(a+b+c) & \zeta_{6}(a+b+c) & \zeta_{3}(a+b+c) \\
  \end{array}
\right)
\end{equation}
Using the relation~(\ref{e6}) and~(\ref{e10}) in~(\ref{e2}) we have
the  result. Thereafter we show that this equation can be reduced to
the following linear system \begin{equation} \label{e11} \frac{d
\eta}{dt}
 = - \eta.
 \end{equation}Indeed, since  \begin{equation}\label{e12}
\dot{\overrightarrow{\eta}}=G\dot{\overrightarrow{\theta}}
\end{equation}
so,  using~(\ref{e8})  and~(\ref{e9}) in~(\ref{e12}) we obtain
\begin{equation*} \left(
                             \begin{array}{c}
                               \dot{\eta}_{1} \\
                               \dot{\eta}_{2} \\
                               \dot{\eta}_{3} \\
                             \end{array}
                           \right)
 =\left(
     \begin{array}{c}
     -\phi_{a}(a) + \phi_{a}(a+b+c)\\
      - \phi_{b}(b) + \phi_{b}(a+b+c)\\
      - \phi_{c}(c) + \phi_{c}(a+b+c)\\
     \end{array}
   \right)
 \end{equation*}
which clearly shows that system~(\ref{e19}) is linearized. Using the
main Theorem in \cite{mama-proceeding}, and using (\ref{e8}), by
setting
\[P_{1}=-\frac{1}{\phi_{a}(a) - \phi_{a}(a+b+c)},\; Q_{1}=\phi_{b}(b)
- \phi_{b}(a+b+c),\] \[P'_{1}=-\frac{1}{\phi_{b}(b) -
\phi_{b}(a+b+c)},\; Q'_{1}= \phi_{c}(c) - \phi_{c}(a+b+c).\], We
have $\mathcal{H}=Q_{1}P_{1}+Q'_{1}P'_{1}$, so
\begin{equation*}
\mathcal{H}(a,b,c)= \frac{\phi_{b}(b) - \phi_{b}(a+b+c)}{\phi_{a}(a)
- \phi_{a}(a+b+c)}+\frac{\phi_{c}(c) - \phi_{c}(a+b+c) }{\phi_{b}(b)
- \phi_{b}(a+b+c) }.
\end{equation*}
We have
\[ \dot{P}_{1}=P_{1},\; \dot{Q}_{1}=-Q_{1},\; \dot{P'}_{1}=P'_{1},\;
\dot{Q'}_{1}=-Q'_{1}. \] Using the Lemma $1$ we have \[
\frac{\partial \mathcal{H}}{\partial Q_{1}}=P_{1},\;
 \frac{\partial \mathcal{H}}{\partial P_{1}}=Q_{1},\;
 \frac{\partial \mathcal{H}}{\partial Q'_{1}}=P'_{1},\;
 \frac{\partial \mathcal{H}}{\partial P'_{1}}=Q'_{1}.
\]   We have the following
system\begin{eqnarray}\label{e16}\left\{
                                     \begin{array}{ll}
                                       \dot{P}_{1} =\frac{\partial \mathcal{H}}{\partial Q_{1}}& \hbox{} \\
                                      \dot{Q}_{1}= -\frac{\partial \mathcal{H}}{\partial P_{1}}& \hbox{} \\
                                        \dot{P'}_{1}= \frac{\partial \mathcal{H}}{\partial Q'_{1}}& \hbox{} \\
                                       \dot{Q'}_{1}=-\frac{\partial \mathcal{H}}{\partial P'_{1}} & \hbox{.}
                                     \end{array}
                                   \right.
\end{eqnarray}

the system~(\ref{e16}) take the form:\begin{eqnarray*}\left(
                    \begin{array}{c}
                      \dot{P}_{1} \\
                      \dot{Q}_{1}\\
                      \dot{P'}_{1}\\
                      \dot{Q'}_{1}\\
                    \end{array}
                  \right)&=& \left(
                              \begin{array}{cccc}
                                0 & 1 & 0 & 0 \\
                                -1 & 0 & 0 & 0 \\
                                0 & 0 & 0 & 1 \\
                                0 & 0 & -1 & 0 \\
                              \end{array}
                            \right)
                  \left(
                                    \begin{array}{c}
                                     \frac{\partial \mathcal{H}}{\partial P_{1}} \\
                                      \frac{\partial \mathcal{H}}{\partial Q_{1}}\\
\frac{\partial \mathcal{H}}{\partial P'_{1}}\\
 \frac{\partial
\mathcal{H}}{\partial Q'_{1}}
                                    \end{array}
                                  \right).\end{eqnarray*}
We see that~(\ref{e9}) is  Hamiltonian system.  Since the
system~(\ref{e9}) is completely integrable on $S$.
\end{proof}
The gradient system defined by the Bivariate Beta family of the
first kind with three parameters odd-dimensional manifold is a
completely integrable Hamiltonian system.
\section{Asymptotic Pseudo-Riemannian structure.}\label{sec5}
 In this section, we examine the geometric properties of the statistical manifold associated with
the asymptotic approximation of the three-parameter first-kind beta
distribution, as described by Stirling's
formula.~\cite{daniel-journal}.
\subsection{Associated information metric and   asymptotic approximation gradient system.}\label{subsec3} Using James
Stirling's formula which was recalled by Mortici,
Cristinel~\cite{daniel-journal}  we have the following proposition
\begin{proposition}Let $(a,b,c) \in ]1;+\infty[\times ]1;+\infty[\times ]1;+\infty[$.
The asymptotic approximation potential function $\Phi$ is defined on
the space of three dimensional deprived of the cube of edges $1cm$
and of which one of the vertices are the origin of the frame and the
three positive axes define the supports of these edges, its
expression is given by
\begin{eqnarray*}
 \Phi(\theta)&=& (a+b+c-\frac{1}{2})\log (a+b+c-1)+(\frac{1}{2}-a)\log (a-1)+(\frac{1}{2}-b)\log (b-1)
 \\
 &&+(\frac{1}{2}-c)\log
 (c-1)-\log(
 2\pi) -2.
\end{eqnarray*}
\end{proposition}
\begin{proof}\textbf{}\newline
Using (\ref{e4}) and Stirling~\cite{daniel-journal}  we have
\begin{eqnarray}\label{e17}
\Phi(\theta)&=& (a+b+c-\frac{1}{2})\log
(a+b+c-1)+(\frac{1}{2}-a)\log (a-1)+(\frac{1}{2}-b)\log
(b-1)\nonumber\\
&&+(\frac{1}{2}-c)\log
 (c-1)-\log(
 2\pi) -2.
\end{eqnarray}
Relation~(\ref{e17}) is defined only for the values of $a$,$b$ ,and
$c$ such that:\[(a,b,c) \in ]1;+\infty[\times ]1;+\infty[\times
]1;+\infty[.\] By setting \[\Upsilon=\left\{(a,b,c)\in
\mathbb{R}^{3} \left|\right. 0 \leq a \leq 1,0 \leq b \leq 1,0 \leq
c \leq 1\right\}.\] $\Upsilon$ defines the cube of edges $1cm$. Then
$(a,b,c) \in \mathbb{R}^{3}-\{\Upsilon\}$.
\end{proof}
So, we have the following proposition
\begin{proposition}
Let $\Upsilon=\left\{(a,b,c)\in \mathbb{R}^{3} \left|\right. 0 \leq
a \leq 1,0 \leq b \leq 1,0 \leq c \leq 1\right\}$  the cube of edges
$1cm$. The asymptotic pseudo-Riemannian structure related to
Stirling's formula  defined for all $(a,b,c) \in
\mathbb{R}^{3}-\{\Upsilon\}$ is given by
\footnotesize{\begin{equation}\label{e18} G= \left(
  \begin{array}{ccc}
    \frac{1}{ a+b+c-1} -\frac{a-\frac{3}{2}}{(a-1)^{2}} &  \frac{1}{ a+b+c-1} & \frac{1}{ a+b+c-1}\\
    \frac{1}{ a+b+c-1} & \frac{1}{ a+b+c-1} -\frac{b-\frac{3}{2}}{(b-1)^{2}} & \frac{1}{ a+b+c-1}\\
    \frac{1}{ a+b+c-1}  & \frac{1}{ a+b+c-1}&  \frac{1}{ a+b+c-1} -\frac{c-\frac{3}{2}}{(c-1)^{2}}  \\
  \end{array}
\right)
\end{equation}}
\end{proposition}
\begin{proof}
The asymptotic approximation potential function (\ref{e17}) we
obtain the following coefficients,
\begin{eqnarray*}
g_{11}( \theta)
&=&\frac{1}{ a+b+c-1} -\frac{a-\frac{3}{2}}{(a-1)^{2}}. \\
g_{22}( \theta)
&=& \frac{1}{ a+b+c-1} -\frac{b-\frac{3}{2}}{(b-1)^{2}}. \\
g_{12}( \theta)
&=&  \frac{1}{ a+b+c-1}.\\
g_{21}( \theta)
&=&   \frac{1}{ a+b+c-1}\\
g_{31}( \theta)
&=& \frac{1}{ a+b+c-1}.\\
g_{32}(\theta)
&=& \frac{1}{ a+b+c-1}.\\
g_{13}(\theta)
&=& \frac{1}{ a+b+c-1}.\\
g_{23}(\theta)&= & \frac{1}{ a+b+c-1}.\\
g_{33}(\theta) &=& \frac{1}{
a+b+c-1}-\frac{c-\frac{3}{2}}{(c-1)^{2}}.
\end{eqnarray*}
 \end{proof}
For all $(a,b,c) \in \mathbb{R}^{3}-\{\Upsilon\}$\newline with
$\Upsilon=\left\{(a,b,c)\in \mathbb{R}^{3} \left|\right. 0 \leq a
\leq 1,0 \leq b \leq 1,0 \leq c \leq 1\right\}$. We have a dual
coordinate $\eta_{i}=\partial_{i} \Phi (\theta)$. So,
\begin{eqnarray*}\left\{
  \begin{array}{ll}
    \eta_{1}=  \log (a+b+c-1)- \log (a-1)+\frac{-\frac{1}{2}}{a-1} & \hbox{} \\
    \eta_{2}=   \log (a+b+c-1)- \log (b-1)+\frac{-\frac{1}{2}}{b-1}& \hbox{} \\
     \eta_{3}=   \log (a+b+c-1)- \log (c-1)+\frac{-\frac{1}{2}}{c-1}& \hbox{}
  \end{array}
\right.\end{eqnarray*} we have the following system
\begin{eqnarray}\label{e19}\left\{
  \begin{array}{ll}
    \frac{\partial}{\partial a}\frac{\partial\varphi(\theta)}{\partial a}=  \frac{1}{ a+b+c-1} -\frac{a-\frac{3}{2}}{(a-1)^{2}} & \hbox{} \\
    \frac{\partial}{\partial a}\frac{\partial\varphi(\theta)}{\partial b}= \frac{1}{ a+b+c-1}& \hbox{} \\
     \frac{\partial}{\partial a}\frac{\partial\varphi(\theta)}{\partial c}= \frac{1}{ a+b+c-1}& \hbox{} \\
    \frac{\partial}{\partial b}\frac{\partial\varphi(\theta)}{\partial b}= \frac{1}{ a+b+c-1} -\frac{b-\frac{3}{2}}{(b-1)^{2}} & \hbox{} \\
     \frac{\partial}{\partial b}\frac{\partial\varphi(\theta)}{\partial c}=  \frac{1}{ a+b+c-1}&\hbox{}\\
     \frac{\partial}{\partial c}\frac{\partial\varphi(\theta)}{\partial c}= \frac{1}{ a+b+c-1} -\frac{c-\frac{3}{2}}{(c-1)^{2}}& \hbox{.}
  \end{array}
\right.\end{eqnarray} we have \begin{eqnarray*}
 \Phi(\theta)&=& (a+b+c-\frac{1}{2})\log (a+b+c-1)+(\frac{1}{2}-a)\log (a-1)+(\frac{1}{2}-b)\log
 (b-1)\\
 &&+(\frac{1}{2}-c)\log
 (c-1)+k.\; k\in \mathbb{R}
\end{eqnarray*}
 We have \begin{equation} \label{e20}\Psi(\eta)= a
\eta_{1}+b \eta_{2}+c \eta_{3}-\Phi(\theta)\end{equation} so, we
have the following proposition
\begin{proposition}
The  dual potential function of the first species beta manifold with
three parameters related to asymptotic approximation is defined for
all $(a,b,c) \in \mathbb{R}^{3}-\{\Upsilon\}$ is given by
\begin{eqnarray*}
\Psi(\eta)&= &  -\frac{ a}{2(a-1)}-\frac{b}{2(b-1)}
 -\frac{c}{2(c-1)}+\frac{1}{2}\log
(a+b+c-1)-\frac{1}{2}\log (a-1)\\
&&-\frac{1}{2}\log (b-1)-\frac{1}{2}\log
 (c-1)-k\end{eqnarray*}
\end{proposition}
\begin{proof}
By applying~(\ref{e20}) we have the result.
\end{proof}
We start by computing the determinant, and we have:
\begin{equation*}det (G)=-\frac{1}{
8}\frac{4cab-8ca+15a-8ba+15c-8bc-27+15b}{
(a-1)^{2}(b-1)^{2}(c-1)^{2}(a+b+c-1)}
\end{equation*}
So, we have the following theorem
\begin{theorem}
Let $\Upsilon$ the cube such that\\ $\Upsilon=\left\{(a,b,c)\in
\mathbb{R}^{3} \left|\right. 0 \leq a \leq 1,0 \leq b \leq 1,0 \leq
c \leq 1\right\}$,\\ let $(a,b,c) \in \mathbb{R}^{3}-\{\Upsilon\}$,
let $D=<A(0,\frac{3}{2},\frac{3}{2});\; \vec{e}\left(
                             \begin{array}{c}
                               1 \\
                                0\\
                               0 \\
                             \end{array}
                           \right)>
$ the vector line,\\ let $\mathbb{V}=< \left(
\frac{8bc-15b-15c+27}{4cb-8b-8c+15},b,c
                           \right)>$
\begin{itemize}
    \item If   $(a,b,c)\in
\mathbb{V}\cup \{D\}$,$\;$  then $S$ then is an algebraic manifold.
    \item If  $(a,b,c)\in \mathbb{R}^{3}-\{\Upsilon;\mathbb{V}\cup \{D\}\}$ then $G$  is an
asymptotic pseudo-Riemannian metric related to Stirling's formula.
\end{itemize}
\end{theorem}
\begin{proof}
For all $(a,b,c) \in \mathbb{R}^{3}-\{\Upsilon\}$
\begin{itemize}
    \item we solve $det (G)=0$ and we have $a = \frac{27-15b-15c+8cb}{4cb-8b-8c+15},\;
b = b,\; c = c$$\;$, where $a = a, \; b = \frac{3}{2}, c
=\frac{3}{2}$. or $a = a, \; b = \frac{3}{2}, c =\frac{3}{2}$ define
the vector line $D$ and $a= \frac{27-15b-15c+8cb}{4cb-8b-8c+15},\; b
= b,\; c = c$ define $\mathbb{V}$. So if $(a,b,c)\in \mathbb{V}\cup
\{D\}$,$\;$ then $S$ then is algebraic manifold.
    \item If $(a,b,c)\in \mathbb{R}^{3}-\{\Upsilon;\mathbb{V}\cup \{D\}\}$ then $det (G)\neq 0$ and $G$ is a symmetric matrix then $G$ is an
     asymptotic pseudo-Riemannian metric.
\end{itemize}
\end{proof}
\begin{theorem}Let $\Upsilon$ the cube of edges $1cm$. Let $(a,b,c)\in
\mathbb{R}^{3}-\{\Upsilon;\mathbb{V}\cup \{D\}\}$.  $G$ is an
asymptotic pseudo-Riemannian metric. Let \begin{equation*}S =
\left\{p_{\theta}(x)=\frac{1}{B(a,b,c)}x_{1}^{a-1}x_{2}^{b-1}(1-x_{1}-x_{2})^{c-1}
,\left.
                        \begin{array}{ll}
                         \theta= (a,\; b,\;c)\in \mathbb{R}^{*}_{+}\times \mathbb{R}^{*}_{+}\times \mathbb{R}^{*}_{+}& \hbox{} \\
                        x=(x_{1},\;x_{2}) \in \mathbb{R}^{*}_{+}\times \mathbb{R}^{*}_{+} & \hbox{}\\
                        x_{1}+x_{2}<1  & \hbox{}\\
                        B(a,b,c)= \frac{\Gamma(a).\Gamma(b).\Gamma(c)}{\Gamma(a+b+c)}& \hbox{}
                        \end{array}
                      \right.
\right\}\end{equation*}
 be the statistical manifold, where
$p_{\theta}$ is the density function of the Bivariate Beta family of
the first kind with three parameters family. The approximated
gradient system according stirling's formula in the asymptotic
regime where $a, b, c > 1$ on $S$ is given by
\begin{eqnarray}\label{e21} \left\{
                \begin{array}{ll}
                  \dot{a} = \frac{2(a-1)^{2}(-6ba-6ca+9a+4cab-c-b+3)}{4cab-8ca+15a-8ba+15c-8bc-27+15b}\left(  \log (a+b+c-1)- \log
                  (a-1)-\frac{\frac{1}{2}}{a-1}\right) \hbox{} \\
                  +\frac{4(2c-3)(b-1)^{2}(a-1)^{2}
}{4cab-8ca+15a-8ba+15c-8bc-27+15b}\left(  \log (a+b+c-1)- \log
(b-1)-\frac{\frac{1}{2}}{b-1}\right)& \hbox{} \\
+\frac{4(2b-3)(c-1)^{2}(a-1)^{2}}{4cab-8ca+15a-8ba+15c-8bc-27+15b}
\left(  \log (a+b+c-1)- \log (c-1)-\frac{\frac{1}{2}}{c-1}\right)& \hbox{} \\
                  \dot{b} =\frac{4(2c-3)(b-1)^{2}(a-1)^{2} }{4cab-8ca+15a-8ba+15c-8bc-27+15b}\left(   \log (a+b+c-1)- \log
                  (a-1)-\frac{\frac{1}{2}}{a-1}\right) \hbox{} \\
                  +\frac{2(b-1)^{2}(-6ba-a+4cab+9b-c-6bc+3)}{4cab-8ca+15a-8ba+15c-8bc-27+15b}
\left(   \log (a+b+c-1)- \log (b-1)-\frac{\frac{1}{2}}{b-1}\right)& \hbox{} \\
+\frac{4(2a-3)(b-1)^{2}(c-1)^{2}}{4cab-8ca+15a-8ba+15c-8bc-27+15b}
\left(  \log (a+b+c-1)- \log (c-1)-\frac{\frac{1}{2}}{c-1}\right)& \hbox{} \\
                  \dot{c} =\frac{4(2b-3)(c-1)^{2}(a-1)^{2}}{4cab-8ca+15a-8ba+15c-8bc-27+15b}\left(  \log (a+b+c-1)- \log
                  (a-1)-\frac{\frac{1}{2}}{a-1}\right) \hbox{} \\
                  +\frac{4(2a-3)(b-1)^{2}(c-1)^{2}}{4cab-8ca+15a-8ba+15c-8bc-27+15b}\left(
\log (a+b+c-1)- \log (b-1)-\frac{\frac{1}{2}}{b-1}\right)&\hbox{}\\
+\frac{2(c-1)^{2}(-6ca-a+4cab+9c+3-b-6bc)}{4cab-8ca+15a-8ba+15c-8bc-27+15b}\left(
\log (a+b+c-1)- \log (c-1)-\frac{\frac{1}{2}}{c-1}\right)& \hbox{.}
                \end{array}\right.
\end{eqnarray}
\end{theorem}
\begin{proof}
In the case, $G$ is an asymptotic pseudo-Riemannian metric, using
proposition $6$ we have the inverse of the matrix will be given by
\begin{equation}\label{e22}
G^{-1}=\left(
  \begin{array}{ccc}
    a1 & a2 & a3 \\
    b1 & b2 & b3 \\
    c1 & c2 & c3 \\
  \end{array}
\right)\\
\end{equation}
with $\left\{
  \begin{array}{ll}
    a1= -\frac{2(a-1)^{2}(-6ba-6ca+9a+4cab-c-b+3)}{4cab-8ca+15a-8ba+15c-8bc-27+15b} & \hbox{} \\
    a2= -\frac{4(2c-3)(b-1)^{2}(a-1)^{2} }{4cab-8ca+15a-8ba+15c-8bc-27+15b} & \hbox{} \\
    a3=-\frac{4(2b-3)(c-1)^{2}(a-1)^{2}}{4cab-8ca+15a-8ba+15c-8bc-27+15b} & \hbox{} \\
     b1= -\frac{4(2c-3)(b-1)^{2}(a-1)^{2} }{4cab-8ca+15a-8ba+15c-8bc-27+15b}& \hbox{} \\
    b2= -\frac{2(b-1)^{2}(-6ba-a+4cab+9b-c-6bc+3)}{4cab-8ca+15a-8ba+15c-8bc-27+15b} & \hbox{} \\
    b3=-\frac{4(2a-3)(b-1)^{2}(c-1)^{2}}{4cab-8ca+15a-8ba+15c-8bc-27+15b} & \hbox{} \\
    c1= -\frac{4(2b-3)(c-1)^{2}(a-1)^{2}}{4cab-8ca+15a-8ba+15c-8bc-27+15b} & \hbox{} \\
    c2=-\frac{4(2a-3)(b-1)^{2}(c-1)^{2}}{4cab-8ca+15a-8ba+15c-8bc-27+15b} & \hbox{} \\
    c3=-\frac{2(c-1)^{2}(-6ca-a+4cab+9c+3-b-6bc)}{4cab-8ca+15a-8ba+15c-8bc-27+15b} & \hbox{.}
  \end{array}
\right.$\\
So, we have
 \begin{equation}\label{e23}\partial_{\theta} \Phi(\theta)=\left(
                                           \begin{array}{c}
                                                \log (a+b+c-1)- \log (a-1)+\frac{-\frac{1}{2}}{a-1}  \\
                                              \log (a+b+c-1)- \log (b-1)+\frac{-\frac{1}{2}}{b-1} \\
                                             \log (a+b+c-1)- \log (c-1)+\frac{-\frac{1}{2}}{c-1}\\
                                           \end{array}
                                         \right)
  .\end{equation}
The  approximated gradient system via stirling's formula will
therefore be written as follows
\begin{eqnarray*} \left\{
                \begin{array}{ll}
                  \dot{a} = \frac{2(a-1)^{2}(-6ba-6ca+9a+4cab-c-b+3)}{4cab-8ca+15a-8ba+15c-8bc-27+15b}\left(  \log (a+b+c-1)- \log
                  (a-1)-\frac{\frac{1}{2}}{a-1}\right) \hbox{} \\
                  +\frac{4(2c-3)(b-1)^{2}(a-1)^{2}
}{4cab-8ca+15a-8ba+15c-8bc-27+15b}\left(  \log (a+b+c-1)- \log
(b-1)-\frac{\frac{1}{2}}{b-1}\right)& \hbox{} \\
+\frac{4(2b-3)(c-1)^{2}(a-1)^{2}}{4cab-8ca+15a-8ba+15c-8bc-27+15b}
\left(  \log (a+b+c-1)- \log (c-1)-\frac{\frac{1}{2}}{c-1}\right)& \hbox{} \\
                  \dot{b} =\frac{4(2c-3)(b-1)^{2}(a-1)^{2} }{4cab-8ca+15a-8ba+15c-8bc-27+15b}\left(   \log (a+b+c-1)- \log
                  (a-1)-\frac{\frac{1}{2}}{a-1}\right) \hbox{} \\
                  +\frac{2(b-1)^{2}(-6ba-a+4cab+9b-c-6bc+3)}{4cab-8ca+15a-8ba+15c-8bc-27+15b}
\left(   \log (a+b+c-1)- \log (b-1)-\frac{\frac{1}{2}}{b-1}\right)& \hbox{} \\
+\frac{4(2a-3)(b-1)^{2}(c-1)^{2}}{4cab-8ca+15a-8ba+15c-8bc-27+15b}
\left(  \log (a+b+c-1)- \log (c-1)-\frac{\frac{1}{2}}{c-1}\right)& \hbox{} \\
                  \dot{c} =\frac{4(2b-3)(c-1)^{2}(a-1)^{2}}{4cab-8ca+15a-8ba+15c-8bc-27+15b}\left(  \log (a+b+c-1)- \log
                  (a-1)-\frac{\frac{1}{2}}{a-1}\right) \hbox{} \\
                  +\frac{4(2a-3)(b-1)^{2}(c-1)^{2}}{4cab-8ca+15a-8ba+15c-8bc-27+15b}\left(
\log (a+b+c-1)- \log (b-1)-\frac{\frac{1}{2}}{b-1}\right)&\hbox{}\\
+\frac{2(c-1)^{2}(-6ca-a+4cab+9c+3-b-6bc)}{4cab-8ca+15a-8ba+15c-8bc-27+15b}\left(
\log (a+b+c-1)- \log (c-1)-\frac{\frac{1}{2}}{c-1}\right)& \hbox{.}
                \end{array}\right.
\end{eqnarray*}
\end{proof}
\subsection{Linearization and Hamiltonian function.}\label{subsec4}
Let's calculate now \begin{equation*}
\dot{\overrightarrow{\eta}}=G\dot{\overrightarrow{\theta}}
\end{equation*}
so, we have
\begin{equation*} \left(
                             \begin{array}{c}
                               \dot{\eta}_{1} \\
                               \dot{\eta}_{2} \\
                               \dot{\eta}_{3} \\
                             \end{array}
                           \right)
 =\left[
  \begin{array}{ccc}
    \frac{1}{ a+b+c-1} -\frac{a-\frac{3}{2}}{(a-1)^{2}} &  \frac{1}{ a+b+c-1} & \frac{1}{ a+b+c-1}\\
    \frac{1}{ a+b+c-1} & \frac{1}{ a+b+c-1} -\frac{b-\frac{3}{2}}{(b-1)^{2}} & \frac{1}{ a+b+c-1}\\
    \frac{1}{ a+b+c-1}  & \frac{1}{ a+b+c-1}&  \frac{1}{ a+b+c-1} -\frac{c-\frac{3}{2}}{(c-1)^{2}}  \\
  \end{array}
                                                                  \right]\left(
                                                                           \begin{array}{c}
                                                                             \dot{a} \\
\dot{b}\\
\dot{c}\\
\end{array} \right)
\end{equation*}
we have\begin{equation*} \left(
                             \begin{array}{c}
                               \dot{\eta}_{1} \\
                               \dot{\eta}_{2} \\
                               \dot{\eta}_{3} \\
                             \end{array}
                           \right)
 =\left(
     \begin{array}{c}
     \log (a+b+c-1)- \log (a-1)-\frac{\frac{1}{2}}{a-1} \\
       \log (a+b+c-1)- \log (b-1)-\frac{\frac{1}{2}}{b-1}\\
       \log (a+b+c-1)- \log (c-1)-\frac{\frac{1}{2}}{c-1}\\
     \end{array}
   \right)
 \end{equation*}
which clearly shows that system~(\ref{e21}) is linearized in the
form
\begin{equation}\label{e24} \frac{d \eta}{dt}
 = - \eta
 \end{equation}
we have the following proposition.
\begin{proposition}Let $\Upsilon$ the cube of edges $1cm$. Let $(a,b,c)\in
\mathbb{R}^{3}-\{\Upsilon;\mathbb{V}\cup \{D\}\}$. Let $S$ be the
statistical manifold. The Hamiltonian of the approximated  gradient
system~(\ref{e21}) according to Stirling formula is given on the
beta family of the first kind manifold is given  by
\begin{eqnarray}\label{e25}
\mathcal{H}(a,b,c)&=& \frac{ \log (a+b+c-1)- \log
(b-1)-\frac{\frac{1}{2}}{b-1}}{ \log (a+b+c-1)- \log
(a-1)-\frac{\frac{1}{2}}{a-1}}\nonumber\\&&+\frac{ \log (a+b+c-1)-
\log (c-1)-\frac{\frac{1}{2}}{c-1} }{\log (a+b+c-1)- \log
(b-1)-\frac{\frac{1}{2}}{b-1} }
\end{eqnarray}
\end{proposition}
\begin{proof} Using the main Theorem in \cite{mama-proceeding}, and using
(\ref{e23}), by setting
\begin{eqnarray*}\left\{
  \begin{array}{ll}
    P_{1}=\frac{1}{\log (a+b+c-1)- \log
(a-1)+\frac{-\frac{1}{2}}{a-1}}, & \hbox{} \\
    Q_{1}=\log (a+b+c-1)- \log
(b-1)+\frac{-\frac{1}{2}}{b-1}, & \hbox{} \\
    P'_{1}=\frac{1}{\log (a+b+c-1)-
\log (b-1)+\frac{-\frac{1}{2}}{b-1}}, & \hbox{} \\
    Q'_{1}=\log (a+b+c-1)-
\log (c-1)+\frac{-\frac{1}{2}}{c-1}. & \hbox{}
  \end{array}
\right. \end{eqnarray*}
 We obtain
\begin{eqnarray*} \mathcal{H}(a,b,c)&=& \frac{ \log
(a+b+c-1)- \log (b-1)-\frac{\frac{1}{2}}{b-1}}{ \log (a+b+c-1)- \log
(a-1)-\frac{\frac{1}{2}}{a-1}}\nonumber\\&&+\frac{ \log (a+b+c-1)-
\log (c-1)-\frac{\frac{1}{2}}{c-1} }{\log (a+b+c-1)- \log
(b-1)-\frac{\frac{1}{2}}{b-1} }
\end{eqnarray*}
 and $\{\mathcal{H},\mathcal{H}\}=0$.
$\mathcal{H}$ is involution.\newline So,  using the system
(\ref{e21})  and Hamiltonian function (\ref{e25}), we have:
\begin{equation*} \frac{d \mathcal{H}}{dt}=\frac{\partial
H}{\partial a}\frac{d a}{dt}+\frac{\partial \mathcal{H}}{\partial
b}\frac{d b}{dt}+\frac{\partial \mathcal{H}}{\partial c}\frac{d
c}{dt}=0\end{equation*} .
\end{proof}
\section{ Integrability of approximated Gradient Systems.}\label{sec6}
In this section we prove the complete integrability of approximated
 gradient system (\ref{e21}). We have the following Theorem:
\begin{theorem}Let $\Upsilon$ the cube of edges $1cm$. Let $(a,b,c)\in
\mathbb{R}^{3}-\{\Upsilon;\mathbb{V}\cup \{D\}\}$. On a statistical
manifold $S$ of dimension $3$ defined by the approximated Bivariate
Beta family of the first kind with three parameters, the
approximated  gradient system obtained is a sub-dynamical system of
$4$-dimensional system and is a completely integrable Hamiltonian
system, such that\begin{equation}\label{e27}\left(
                    \begin{array}{c}
                      \dot{P}_{1} \\
                      \dot{Q}_{1}\\
                      \dot{P'}_{1}\\
                      \dot{Q'}_{1}\\
                    \end{array}
                  \right)= \left(
                              \begin{array}{cccc}
                                0 & 1 & 0 & 0 \\
                                -1 & 0 & 0 & 0 \\
                                0 & 0 & 0 & 1 \\
                                0 & 0 & -1 & 0 \\
                              \end{array}
                            \right)
                  \left(
                                    \begin{array}{c}
                                     \frac{\partial \mathcal{H}}{\partial P_{1}} \\
                                      \frac{\partial \mathcal{H}}{\partial Q_{1}}\\
\frac{\partial \mathcal{H}}{\partial P'_{1}}\\
 \frac{\partial
\mathcal{H}}{\partial Q'_{1}}
                                    \end{array}
                                  \right)\end{equation}
where $\barwedge=\left(
                              \begin{array}{cccc}
                                0 & 1 & 0 & 0 \\
                                -1 & 0 & 0 & 0 \\
                                0 & 0 & 0 & 1 \\
                                0 & 0 & -1 & 0 \\
                              \end{array}
                            \right)$ is a Poisson tensor.
\end{theorem}
\begin{proof}We have the following relation.
\begin{eqnarray}\label{p1}\left\{
                            \begin{array}{ll}
                              \frac{\partial P_{1}}{\partial a}=-\frac{\frac{1}{a+b+c-1}
-\frac{2}{a-1}-\frac{\frac{1}{2}-a}{\left(a-1\right)^{2}}}{\left(\log
(a+b+c-1)+1-\log (a-1)+\frac{\frac{1}{2}-a}{a-1}\right)^{2}}, & \hbox{} \\
                              \frac{\partial P_{1}}{\partial b}=\frac{\partial P_{1}}{\partial
c}=-\frac{1}{\left(\log (a+b+c-1)+1-\log
(a-1)+\frac{\frac{1}{2}-a}{a-1}\right)^{2}(a+b+c-1)}, & \hbox{} \\
                              \frac{\partial Q_{1}}{\partial
b}=\frac{1}{a+b+c-1}
-\frac{2}{b-1}-\frac{\frac{1}{2}-b}{\left(b-1\right)^{2}}, & \hbox{} \\
                              \frac{\partial Q_{1}}{\partial a}=\frac{\partial Q_{1}}{\partial
c}=\frac{1}{a+b+c-1}, & \hbox{} \\
                              \frac{\partial P'_{1}}{\partial b}=-\frac{\frac{1}{a+b+c-1}
-\frac{2}{b-1}-\frac{\frac{1}{2}-b}{\left(b-1\right)^{2}}}{\left(\log
(a+b+c-1)+1-\log (b-1)+\frac{\frac{1}{2}-b}{b-1}\right)^{2}}, & \hbox{} \\
                              \frac{\partial P'_{1}}{\partial a}=\frac{\partial P'_{1}}{\partial
c}=-\frac{1}{\left(\log (a+b+c-1)+1-\log
(b-1)+\frac{\frac{1}{2}-b}{b-1}\right)^{2}(a+b+c-1)} & \hbox{.}\\
\frac{\partial Q'_{1}}{\partial
b}=\frac{\partial Q'_{1}}{\partial b}=\frac{1}{a+b+c-1},& \hbox{} \\
\frac{\partial Q'_{1}}{\partial c}= \frac{1}{a+b+c-1}
-\frac{2}{c-1}-\frac{\frac{1}{2}-c}{\left(c-1\right)^{2}}.& \hbox{}
                            \end{array}
                          \right.
\end{eqnarray}
So,  using the system (\ref{e21})  and the relation (\ref{p1}), we
have
\begin{equation*}
\left\{
  \begin{array}{ll}
    \dot{P}_{1}=\frac{\partial P_{1}}{\partial a}\frac{d
a}{dt}+\frac{\partial P_{1}}{\partial b}\frac{d
b}{dt}+\frac{\partial P_{1}}{\partial c}\frac{d c}{dt}=\frac{1}{\log
(a+b+c-1)- \log (a-1)+\frac{-\frac{1}{2}}{a-1}}=P_{1}, & \hbox{} \\
    \dot{Q}_{1}=\frac{\partial Q_{1}}{\partial a}\frac{d
a}{dt}+\frac{\partial Q_{1}}{\partial b}\frac{d
b}{dt}+\frac{\partial Q_{1}}{\partial c}\frac{d c}{dt}=-\left(\log
(a+b+c-1)- \log (b-1)+\frac{-\frac{1}{2}}{b-1}\right)=-Q_{1}, & \hbox{} \\
    \dot{P'}_{1}=\frac{\partial P'_{1}}{\partial a}\frac{d
a}{dt}+\frac{\partial P'_{1}}{\partial b}\frac{d
b}{dt}+\frac{\partial P'_{1}}{\partial c}\frac{d
c}{dt}=\frac{1}{\log (a+b+c-1)- \log
(b-1)+\frac{-\frac{1}{2}}{b-1}}=P'_{1}, & \hbox{} \\
    \dot{Q'}_{1}=\frac{\partial Q'_{1}}{\partial a}\frac{d
a}{dt}+\frac{\partial Q'_{1}}{\partial b}\frac{d
b}{dt}+\frac{\partial Q'_{1}}{\partial c}\frac{d c}{dt}=-\left(\log
(a+b+c-1)- \log (c-1)+\frac{-\frac{1}{2}}{c-1}\right)=-Q'_{1}. &
\hbox{}
  \end{array}
\right.
\end{equation*}
we know that the Hamiltonian function is given by
\begin{equation*}\mathcal{H}(a,b,c)=P_{1}Q_{1}+P'_{1}Q'_{1}\end{equation*}
we have the relation  $ \frac{\partial \mathcal{H}}{\partial
Q_{1}}=P_{1},\;
 \frac{\partial \mathcal{H}}{\partial P_{1}}=Q_{1},\;
 \frac{\partial \mathcal{H}}{\partial Q'_{1}}=P'_{1},\;
 \frac{\partial \mathcal{H}}{\partial P'_{1}}=Q'_{1}
$\newline   We have the following system:\begin{eqnarray}\label{e26}
\left\{
  \begin{array}{ll}
    \frac{\partial \mathcal{H}}{\partial Q_{1}}=\dot{P}_{1} & \hbox{} \\
    -\frac{\partial \mathcal{H}}{\partial P_{1}}=\dot{Q}_{1} & \hbox{} \\
    \frac{\partial \mathcal{H}}{\partial Q'_{1}}=\dot{P'}_{1} & \hbox{} \\
    -\frac{\partial \mathcal{H}}{\partial P'_{1}}=\dot{Q'}_{1} & \hbox{.}
  \end{array}
\right.
\end{eqnarray}
the system~(\ref{e26}) takes the form\begin{equation*}\left(
                    \begin{array}{c}
                      \dot{P}_{1} \\
                      \dot{Q}_{1}\\
                      \dot{P'}_{1}\\
                      \dot{Q'}_{1}\\
                    \end{array}
                  \right)= \left(
                              \begin{array}{cccc}
                                0 & 1 & 0 & 0 \\
                                -1 & 0 & 0 & 0 \\
                                0 & 0 & 0 & 1 \\
                                0 & 0 & -1 & 0 \\
                              \end{array}
                            \right)
                  \left(
                                    \begin{array}{c}
                                     \frac{\partial \mathcal{H}}{\partial P_{1}} \\
                                      \frac{\partial \mathcal{H}}{\partial Q_{1}}\\
\frac{\partial \mathcal{H}}{\partial P'_{1}}\\
 \frac{\partial
\mathcal{H}}{\partial Q'_{2}}.
                                    \end{array}
                                  \right)\end{equation*}
We show that~(\ref{e27}) is equivalent to the Hamiltonian system
\begin{eqnarray}\label{e28}
\left\{
  \begin{array}{ll}
    \frac{\partial \mathcal{H}}{\partial Q_{1}}=\dot{P}_{1} & \hbox{} \\
    \frac{\partial \mathcal{H}}{\partial P_{1}}=-\dot{Q}_{1} & \hbox{} \\
    \frac{\partial \mathcal{H}}{\partial Q'_{1}}=\dot{P'}_{1} & \hbox{} \\
    \frac{\partial \mathcal{H}}{\partial P'_{1}}=-\dot{Q'}_{1} & \hbox{.}
  \end{array}
\right.
\end{eqnarray}
on $S$. Thus $(Q_{1}, P_{1},Q'_{1},P'_{1})$ is a set of canonical
variables. Since $\mathcal{H}\in C^{1}(S)$ does not depend on $t$
explicitly. According to Liouville-Arnol'd~\cite{
Li-book2}system~(\ref{e21}) the complete integrability is proven.
\end{proof}
\section{ Lax pair representation.}\label{sec7}
In \cite{mama0-journal}, it is given the definition of the Lax-pair.
In this section, we will construct the pair of lax associated with
the gradient system. We recall that the pair of Lax must respect two
important conditions: $tr\left(L\right)=\mathcal{H}$ and $\frac{d
tr\left(L\right)}{dt}=0$. With $\frac{d tr\left(L\right)}{dt}=tr
\left[L, N\right]$.\newline So, we have the  Theorem
\begin{theorem}Let $(a,b,c)\in
\mathbb{R}^{3}-\{\Upsilon;\mathbb{V}\cup \{D\}\}$. The approximated
 gradient system on the manifold $S$ of Beta family of the first kind
Distribution is represented by the following Lax
pair:\begin{equation} \dot{L}=  \left[L, N\right]
\end{equation}
\begin{equation*}L=  \left(
                     \begin{array}{ccc}
                       \frac{ \log (a+b+c-1)- \log (b-1)-\frac{\frac{1}{2}}{b-1}}{ \log
(a+b+c-1)- \log (a-1)-\frac{\frac{1}{2}}{a-1}} &0  &\sqrt{ \frac{
\log (a+b+c-1)- \log (c-1)-\frac{\frac{1}{2}}{c-1}}{ \log
(a+b+c-1)- \log (a-1)-\frac{\frac{1}{2}}{a-1}}}  \\
 0  & 0 &0  \\
 \sqrt{ \frac{
\log (a+b+c-1)- \log (c-1)-\frac{\frac{1}{2}}{c-1}}{ \log (a+b+c-1)-
\log (a-1)-\frac{\frac{1}{2}}{a-1}}}  & 0  & \frac{ \log (a+b+c-1)-
\log (c-1)-\frac{\frac{1}{2}}{c-1} }{\log
(a+b+c-1)- \log (b-1)-\frac{\frac{1}{2}}{b-1} } \\
                     \end{array}
                   \right)
\end{equation*} and
 \begin{equation*}N= \left(
                     \begin{array}{ccc}
                       \ell & 0 & 0 \\
                       0 & 0 & 0 \\
                       0 & 0 & \ell\\
                     \end{array}
                   \right)
\end{equation*}
with, $\ell\in \mathbb{R}$.
\end{theorem}
\begin{proof}
we are given a symmetric matrix $L$ such that its trace is equal to
the Hamiltonian and whose unknowns are on the second diagonal as a
sequence.
\begin{equation}\label{e29}L= \left(
                    \begin{array}{ccc}
                       \frac{ \log (a+b+c-1)- \log (b-1)-\frac{\frac{1}{2}}{b-1}}{ \log
(a+b+c-1)- \log (a-1)-\frac{\frac{1}{2}}{a-1}} & 0 & m_{1}\\
                      0 & 0 & 0\\
                      m_{1} &0 & \frac{ \log
(a+b+c-1)- \log (c-1)-\frac{\frac{1}{2}}{c-1} }{\log
(a+b+c-1)- \log (b-1)-\frac{\frac{1}{2}}{b-1} }\\
                    \end{array}
                  \right)
\end{equation}
The matrix~(\ref{e29}) is a symmetric matrix and real-valued, so it
is diagonalizable. And according to Zakharov shabat the determinant
of this matrix is a prime integral therefore its eigenvalues are
prime integrals.   Thus to determine the values of $m$, we pose the
condition that these values are prime integrals. We have the
following matrix
\begin{equation*}L=  \left(
                     \begin{array}{ccc}
                       \frac{ \log (a+b+c-1)- \log (b-1)-\frac{\frac{1}{2}}{b-1}}{ \log
(a+b+c-1)- \log (a-1)-\frac{\frac{1}{2}}{a-1}} &0  &\sqrt{ \frac{
\log (a+b+c-1)- \log (c-1)-\frac{\frac{1}{2}}{c-1}}{ \log
(a+b+c-1)- \log (a-1)-\frac{\frac{1}{2}}{a-1}}}  \\
 0  & 0 &0 \\
 \sqrt{ \frac{
\log (a+b+c-1)- \log (c-1)-\frac{\frac{1}{2}}{c-1}}{ \log (a+b+c-1)-
\log (a-1)-\frac{\frac{1}{2}}{a-1}}}  & 0  & \frac{ \log (a+b+c-1)-
\log (c-1)-\frac{\frac{1}{2}}{c-1} }{\log
(a+b+c-1)- \log (b-1)-\frac{\frac{1}{2}}{b-1} } \\
                     \end{array}
                   \right)
\end{equation*}
The matrix that allows us to have $\frac{d tr\left(L\right)}{dt}=0$
is given by
\begin{equation*}N= \left(
                    \begin{array}{ccc}
                      \ell_{1} & 0 & 0 \\
                      0 & 0 & 0 \\
                      0 & 0 & \ell_{2} \\
                    \end{array}
                  \right)
\end{equation*}
To determine the values of $\ell_{1}$ and $\ell_{2}$, we solve the
equation: $\dot{L}= \left[L, N\right]$. So we obtain
\begin{equation*}N= \left(
                    \begin{array}{ccc}
                      \ell & 0 & 0 \\
                      0 & 0 & 0 \\
                      0 & 0 & \ell \\
                    \end{array}
                  \right)
\end{equation*} for all $\ell \in\mathbb{R}$.
So, we have  \begin{equation*}\frac{dL}{dt}= \left(
                    \begin{array}{ccc}
                     0 & 0 & 0 \\
                      0 & 0 & 0 \\
                      0 & 0 & 0\\
                    \end{array}
                  \right) \end{equation*}
and we have
\begin{equation*}\left[L, N\right]=\left(
                    \begin{array}{ccc}
                     0 & 0 & 0 \\
                      0 & 0 & 0 \\
                      0 & 0 & 0\\
                    \end{array}
                  \right)
                  \end{equation*}
\end{proof}
\section{General conclusion}\label{sec8}
In this paper, we have studied the gradient system associated with
the tree-parameter bivariate beta statistical manifold of the first
kind. We showed that this system admits a potential function and a
Fisher information metric, and that the corresponding dynamics can
be expressed as a Hamiltonian system. Using Stirling's approximation
for the gamma function, we derived an explicit and simplified form
of the potential function, allowing us to construct an approximated
Fisher information matrix and inverse. This led to a linearized
gradient system,for which we obtained a Poisson structure and a
Hamiltonian in involution. The complete integrability demonstrated
in this paper holds for the asymptotic approximation of the
bivariate beta statistical manifold, valid in the regime where
$a,b,c>1$. The use of Stirling's formula simplifies the geometry but
implies that the results concern an approximate model. Therefore,
the Hamiltonian structure and complete integrability apply to the
approximated system derived  from this asymptotic setting, and not
to the exact bivariate beta model in its full generality.



\subsection*{Acknowledgments.} We would like to thank all the active
members of  the University of Yaounde1 in Cameroon. We would like to
thank all the active members of the algebra and geometry research
group at the University of Maroua in Cameroon. We would also like to
thank the laboratory of the Higher National School of Polytechnic of
Yaounde 1 in Cameroon.




\EndPaper



\begin{thebibliography}{99}
\bibitem{morris-book4}M. W. Hirsch,S. Smale, and R.L. Devaney,  \emph{Differential Equations,
Dynamical Systems, and an Introduction to Chaos}. Elsevier Academic
Press, The United States Of America (2004)
\bibitem{mama-proceeding}P.R.  Mama Assandje, J. Dongho, and T. Bouetou Bouetou,  \emph{Complete integrability of
gradient systems on a manifold admitting a potential in odd
dimension}, Geometric Science of Information 23, 23 (2023)
https://doi.org/10.1007/978-3-031-38299 44
\bibitem{Shu-book1}S. Amari, and H. Nagaoka,  \emph{Methods of Information Geometry },vol. 191,
(2000). https://doi.org/10.1090/mmono/191
\bibitem{souriau} J.M. Souriau, \emph{ Structure des systemes dynamiques (paris: Dunod).
Translated as Structure of Dynamical Systems: A Symplectic view of
Physics} (1969)
\bibitem{nak-journal} Y. Nakamura, \emph{Completly integrable systems on the manifolds of
gaussian and multinomial distribution}, japan journal of industrial
and applied mathematics 10, 179--189 (1993)
https://doi.org/10.1007/BF03167571
\bibitem{nakamura3} Y. Nakamura,  \emph{Gradient
systems associated with probability distributions}, Japan journal of
industrial and applied mathematics 11, 21--30 (1994)
\bibitem{nakamura4} Y. Nakamura, and L. Faybusovich, \emph{On explicitly solvable gradient
systems of moser karmarkar type}. Journal of Mathematical Analysis
and Applications 205(1), 88--106 (1997)
\bibitem{nakamura1} Y. Nakamura, \emph{Lax pair and
fixed point analysis of karmarkars projective scaling trajectory for
linear programming}. Japan journal of industrial and applied
mathematics 11, 1--9 (1994)
\bibitem{nakamura2} Y. Nakamura,  \emph{Neurodynamics and nonlinear integrable
systems of lax type}, Japan journal of industrial and applied
mathematics 11(1), 11--20 (1994)
\bibitem{Dy-journal}A. Fujiwara, \emph{Dynamical systems on statistical models (state of art
and perspectives of studies on nonliear integrable systems)}. RIMS
Kkyuroku 822, 32--42 (1993)
https://doi.org/http://hdl.handle.net/2433/83219
\bibitem{Dy1}A. Fujiwara, and S.
Shuto., \emph{Hereditary structure in hamiltonians: Information
geometry of ising spin chains}, Physics Letters A 374(7), 911--916
(2010)
\bibitem{jpf} J.P. Francois, \emph{Information geometry and integrable
hamiltonian systems. In: workshop on Joint Structures and Common
Foundations of Statistical Physics, Information Geometry and
Inference for Learning}, pp. 141--153 (2020). Springer
\bibitem{mama0-journal} P.R.  Mama Assandje, J. Dongho, and T. Bouetou Bouetou, B.T, \emph{On the complete
integrability of gradient systems on manifold of the lognormal
family}, Chaos, Solitons \& Fractals 173, 113695 (2023)
https://doi.org/10.1016/j.chaos.2023.113695
\bibitem{mama1} P.R.  Mama Assandje, and J. Dongho, \emph{Geometric properties of beta distributions}, Geometric
Science of Information 23, 44 (2023)
https://doi.org/10.1007/978-3-031-38271- 0 21
\bibitem{mama24}P.R.  Mama Assandje, J. Dongho, and T. Bouetou Bouetou, B.T, \emph{On the complete
integrability of gradient systems on manifold of the beta family of
the first kind}. Information Geometry, 1--25 (2024)
https://doi.org/10.1007/s41884-023-00130-z
\bibitem{daniel-journal}C. Mortici,  \emph{Ramanujan formula for the generalized stirling
approximation}, Applied mathematics and computation 217(6),
2579--2585 (2010)
\bibitem{ovidiu-book3} C.Ovidiu,  \emph{Constantin, U.: Geometric Modeling in
Probability and Statistic}, vol. 121. Springer, USA,Romania (2014).
https://doi.org/10.1007/978-3-319-07779-6
\bibitem{lesfari2}A. Lesfari, \emph{Theorie spectrale et probleme non-lineaires}, Surv.
Math. Appl. 5, 141--180 (2010)
\bibitem{Li}V.I. Arnol'd,  V. Arnold,\emph{
Symplectic Geometry vol. 191}, Springer, Usa (1990).
https://doi.org/10.1007/978-3-662-06791-8
\bibitem{Li-book2}M.K. Slifka, and
J.L.Whitton, \emph{ Symplectic geometry. Dynamical systems IV}, 191,
1--138 (2001) https://doi.org/10.1007/978-3-662-06791-8
\end{thebibliography}
\end{document}